\documentclass[11pt]{amsart}

\usepackage{amsmath}
\usepackage{amssymb}
\usepackage{fancybox}
\usepackage{eucal}
\usepackage{amsthm}
\usepackage{amscd}
\usepackage{wasysym}
\usepackage{url}
\usepackage{MnSymbol} 

\usepackage[all,cmtip]{xy}

\usepackage{enumitem}

\usepackage{natbib}

\usepackage{hyperref}

\usepackage{changepage}

\newcommand{\comment}[1]{}

\DeclareMathSymbol{\mlq}{\mathord}{operators}{``}
\DeclareMathSymbol{\mrq}{\mathord}{operators}{`'}

\DeclareMathOperator{\gal}{Gal}

\newtheorem {thm}{Theorem}[section]
\newtheorem{corollary}[thm]{Corollary}

\newtheorem{fact}[thm]{Fact}

\newtheorem{prop}[thm]{Proposition}

\newtheorem {lem}[thm]{Lemma}
\newtheorem*{lemstar}{Lemma}

\newtheorem{ques}[thm]{Question}

\newtheorem{claim}{Claim}
\newtheorem*{claimstar}{Claim}
\theoremstyle{remark}
\newtheorem{rem}[thm]{Remark}
\newtheorem{np*}{Non-Proof}

\theoremstyle{definition}
\newtheorem{defn}[thm]{Definition}

\newtheorem{exam}[thm]{Example}

\newtheorem{case}{Case}

\numberwithin{subcase}{case}

\newcommand{\mc}{\mathcal }

\begin{document}

\setlist[enumerate]{noitemsep, topsep=0pt}

\title{On computable field embeddings and difference closed fields}

\author[M. Harrison-Trainor]{Matthew Harrison-Trainor}
\address{Group in Logic and the Methodology of Science\\
University of California, Berkeley\\
 USA}
\email{matthew.h-t@berkeley.edu}
\urladdr{\href{http://www.math.berkeley.edu/~mattht/index.html}{www.math.berkeley.edu/$\sim$mattht}}

\author[A. Melnikov]{Alexander Melnikov}
\address{The Institute of Natural and Mathematical Sciences\\
Massey University\\
 New Zealand}
\email{alexander.g.melnikov@gmail.com}
\urladdr{\href{https://dl.dropboxusercontent.com/u/4752353/Homepage/index.html}{https://dl.dropboxusercontent.com/u/4752353/Homepage/index.html}}

\author[R. Miller]{Russell Miller}
\address{Dept.\ of Mathematics, Queens College, \& Ph.D.\ Programs in Mathematics \& Computer Science,
Graduate Center, City University of New York,  USA}
\email{Russell.Miller@qc.cuny.edu}
\urladdr{\href{http://qcpages.qc.cuny.edu/~rmiller/}{http://qcpages.qc.cuny.edu/$\sim$rmiller/}}

\thanks{The first author was partially supported by the Berkeley Fellowship and NSERC grant PGSD3-454386-2014. The second author was partially supported by the Packard Foundation. The third author was supported by NSF grants \# DMS-1362206 and DMS-1001306, and by several PSC-CUNY research awards. Some of this work took place at a workshop held by the Institute for Mathematical Sciences of the National University of Singapore.}

\begin{abstract}
We investigate when a computable automorphism of a computable field can be effectively extended to a computable automorphism of its (computable) algebraic closure. We then apply our results and techniques to study effective embeddings of computable difference fields into computable difference closed fields.
\end{abstract}


\maketitle


\section{Introduction}
\label{Introduction}

This article is a contribution to effective field theory, where the main objects of study are computable fields. Recall that an algebraic structure is computable if the elements of its domain are associated with natural numbers in such a way that the operations become computable functions upon this domain \cite{Malcev61,Rabin60}.
There are a number of classical results which say that maps between fields can be extended to maps between their algebraic closures. We consider when this can be done effectively.  That is, if all of the fields involved are computable, and we are given a computable map, must there exist a computable extension to the algebraic closures? 
We obtain both necessary and sufficient conditions on a computable field $\mc{F}$ which ensure that these classical theorems hold effectively for the field $\mc{F}$. 
We also apply our results to computable fields with a distinguished (computable) automorphism; such fields are known as difference fields. We investigate the problem of effectively embedding difference fields into computable difference-closed fields (these are existentially closed difference fields, to be discussed). As we will see,  the most naive analogy of the well-known results of Rabin~\cite{Rabin60} and Harrington~\cite{Harrington74} fails for computable difference fields, in all characteristics. Nonetheless,  we will find a broad class of fields (including abelian extensions of a prime field) for which a stronger version of the analogous result holds. 

\subsection{Embeddings into algebraically closed fields} In the pioneering paper  \cite{Rabin60},
Rabin proved that every computable field $\mc{F}$ can be embedded into a computable presentation $\mc{E}$ of its algebraic closure by a computable map $\imath \colon \mc{F} \to \mc{E}$. Provided that $\mc{E}$ is algebraic over the image $\imath(\mc{F})$, we call such an embedding $\imath$ a \textit{Rabin embedding} of $\mc{F}$ into $\mc{E}$, writing $\overline{\mc{F}}$ for $\mc{E}$ since $\mc{E}$ may thus be regarded as an algebraic closure of $\mc{F}$. In what follows it will be important that, in general,  the image of $\mc{F}$ under the Rabin embedding $\imath$ does not have to be a computable subset of $\overline{\mc{F}}$.
  Rabin~\cite{Rabin60} showed that the problem of  deciding the $\imath$-image of $\mc{F}$ in $\overline{\mc{F}}$ is fully captured by the notion of the \textit{splitting set}. Recall that the \textit{splitting set} $S_\mc{F}$ of $\mc{F}$ is the set of all polynomials $p \in \mc{F}[X]$ which are reducible over $\mc{F}$. If the splitting set of $\mc{F}$ is computable, then we say that $\mc{F}$ has a \textit{splitting algorithm}.  Rabin~\cite{Rabin60} showed that for each computable field $\mc{F}$, and for each Rabin embedding $\imath$ of $\mc{F}$, the image $\imath(\mc{F})$ of $\mc{F}$ in $\overline{\mc{F}}$ is Turing equivalent to  the splitting set of $\mc{F}$, which may be undecidable~\cite{Rabin60}. We note that splitting algorithms had been studied long before Rabin. For instance, in 1882, Kronecker \cite{Kronecker82}  analyzed splitting algorithms for finitely generated extensions of $\mathbb{Q}$.

\subsection{The first main result} It is well known that every isomorphic embedding $\alpha$ of a field $\mc{F}$ into an algebraically closed $\mc{K}$ extends to an embedding $\beta$ of the algebraic closure of $\mc{F}$ into $\mc{K}$. Since we are interested in \emph{effective} embeddings, we ask whether $\beta$ can always be chosen to be \emph{effective}. In our notation, with a fixed Rabin embedding $\imath$ and an arbitrary computable $\alpha$, we ask for a computable $\beta$ such that the following diagram commutes:
$$\xymatrix{
\overline{\mc{F}}\ar@{-->}[r]^{\beta} & \mc{K}\\
\mc{F}\ar[u]^{\imath}\ar[ur]_\alpha & 
}$$
i.e., $\alpha = \beta \circ \imath$.
If a computable  solution to the diagram above exists for every choice of $\alpha$ and of the computable algebraically closed field $\mc{K}$, then we say that $(\mc{F},\imath)$ has the \emph{computable extendability of embeddings property}.
Notice, however, that if some Rabin embedding $\imath$ of a particular $\mc{F}$
has the computable extendability of embeddings property, then so does every other Rabin embedding $\jmath$ of $\mc{F}$ (into any
computable presentation of $\overline{\mc{F}}$):  just apply the computable extendability of embeddings property for $\imath$,
with $\jmath$ as the $\alpha$, to get an embedding $\beta_{\jmath}$ which extends
$\jmath\circ\imath^{-1}$ (and must be an isomorphism).
Then, given any other $\alpha$, the computable extendability of embeddings property for $\imath$ yields a $\beta$ such that
$\beta\circ\beta_{\jmath}^{-1}$ satisfies the computable extendability of embeddings property for $\jmath$ and this $\alpha$.
Therefore, we usually simply say that $\mc{F}$ itself has the computable extendability of embeddings property.

 The first problem that we address in the paper is:

\medskip

\begin{adjustwidth}{2.5em}{0pt}
\emph{Find a necessary and sufficient condition for a computable $\mc{F}$ to have the computable extendability of embeddings property.}
\end{adjustwidth}

\smallskip

Before we give a necessary and sufficient condition, we discuss a subtlety that would not occur in the classical case. The desired extension $\beta$ clearly depends on the choice of the Rabin embedding $\imath$. Classically, the dependence on $\imath$ is often suppressed, since we can identify $\mc{F}$ with its $\imath$-image. However, as noted above, such an identification is generally impossible \emph{effectively}: the membership problem for $\imath(\mc{F})$ may be undecidable. 
 To emphasize the dependence on the embedding $\imath: \mc{F} \to \overline{\mc{F}}$, we say that $\beta$ \textit{$\imath$-extends} $\alpha$ if it is a solution to the diagram above. Later in the paper we will allow $\imath$ to vary, but for now we fix a concrete choice of a Rabin embedding $\imath$.

We may further restrict ourselves and ask for a \emph{uniform procedure} (i.e., a Turing functional) that takes the open diagram of an algebraically closed field $\mc{K}$ and an embedding $\alpha \colon \mc{F} \to \mc{K}$ and  outputs an embedding  of $\overline{\mc{F}}$ into $\mc{K}$ $\imath$-extending $\alpha$. 
For uniform extendability we do not require $\mc{K}$ or $\alpha$ to be computable, but we still fix $\imath$. The reader may find it somewhat unexpected that this uniform version is equivalent to the computable extendability of embeddings property:

\begin{thm}\label{thm:1}
Let $\mc{F}$ be a computable field together with a computable embedding $\imath\colon \mc{F} \to \overline{\mc{F}}$ of $\mc{F}$ into its algebraic closure. Then the following are equivalent:
\begin{enumerate}
	\item $\mc{F}$ has a splitting algorithm,
	\item $\mc{F}$ has the computable extendability of embeddings property,
	\item There exists a Turing functional which, given as its oracle the open diagram of an algebraically closed field $\mc{K}$ and an embedding $\alpha \colon \mc{F} \to \mc{K}$, computes an embedding  of $\overline{\mc{F}}$ into $\mc{K}$ $\imath$-extending $\alpha$.
\end{enumerate}
\end{thm}

The property captured by Theorem~\ref{thm:1} above is also equivalent to an \emph{a priori} weaker uniform extendability condition, namely the existence of a uniform procedure that takes indices of computable $\mc{K}$ and $\alpha: \mc{F} \rightarrow \mc{K}$ and outputs an index of a computable $\beta: \overline{\mc{F}} \rightarrow \mc{K}$ extending $\alpha$. Indeed, this weaker uniform property \emph{follows} from the uniform extendability condition in Theorem~\ref{thm:1} and \emph{implies} the computable extendability property.

In the language of reverse mathematics, Theorem~\ref{thm:1} would say that in the $\omega$-model $REC$ consisting of the computable sets, a field has a unique algebraic closure if and only if that field has a splitting algorithm. Thus, while $RCA_0$ proves that every field with a splitting algorithm has a unique algebraic closure, it is consistent that every other field has more than one algebraic closure.
We note that it was already known from work in reverse mathematics (and is easy to see) that in the situation described above there is always a \textit{low} $\imath$-extension of $\alpha$, and in characteristic zero if $\mc{F}$ has a splitting algorithm then there is a computable extension of $\alpha$ (see \cite[Theorem 9]{DoraisHirstShafer13} and \cite[Theorem 3.3]{FriedmanSimpsonSmith83}). 
In our result we do not restrict ourselves to fields of characteristic $0$; the issue that we face in the case of a positive characteristic will be circumvented using purely inseparable extensions (to be defined).
We remark that the essential part of  our proof of Theorem~\ref{thm:1} is based  on a certain preservation strategy combined with a variation of the Henkin construction; such a combination has not yet been seen in effective algebra.

\subsection{The second main result}

Another classical result says that every automorphism of a field $\mc{F}$ extends
to an automorphism of its algebraic closure. 
In our notation, the diagram
$$\xymatrix{
\mc{F}\ar[r]^{\imath}\ar[d]_\alpha & \overline{\mc{F}}\ar@{-->}[d]^\beta\\
\mc{F}\ar[r]^{\imath} & \overline{\mc{F}}
}$$
always has a solution $\beta$ such that the diagram commutes, i.e., $\imath \circ \alpha = \beta \circ \imath$. Once again this is dependent on the embedding $\imath: \mc{F} \to \overline{\mc{F}}$, and slightly abusing our terminology we say that $\beta$ \textit{$\imath$-extends} $\alpha$. We ask when $\beta$ can be computed effectively. In the setting of automorphisms, it is natural to look at normal algebraic extensions of the prime field (as we will see in Proposition \ref{fact:rigid}). In this case, we  can apply Theorem~\ref{thm:1} to fully characterize existence of such $\imath$-extensions in terms of a splitting algorithm; 
the exact statement will be given in \S 4  (Corollary~\ref{thm:auto}). 
Although the reader may find Corollary~\ref{thm:auto} interesting on its own right, the discussed above dependence on $\imath$ makes it somewhat unsatisfying. Also, as we will discuss in the next subsection, we would like to apply our results to difference fields, and there this dependence on $\imath$ is an obstacle. Therefore, in contrast to the situation of the computable extendability of embeddings property above, we would like to allow the embedding $\imath$ to vary. 
\begin{defn}\label{def:3}
We say that a computable field $\mc{F}$ has the \emph{computable extendability of automorphisms property} if for every computable automorphism $\alpha \colon \mc{F} \rightarrow \mc{F}$ there is a Rabin embedding $\imath \colon \mc{F} \to \overline{\mc{F}}$ and a computable automorphism $\beta \colon \overline{\mc{F}} \rightarrow \overline{\mc{F}}$ which $\imath$-extends $\alpha$.
\end{defn}

\noindent The second problem we address in the paper is:

\medskip

\begin{adjustwidth}{2.5em}{0pt}
\emph{Find a necessary and sufficient condition for a computable $\mc{F}$ to have the computable extendability of automorphisms property.}
\end{adjustwidth}

\smallskip

As we mentioned above, the computable extendability of automorphisms property is the property which is of interest in constructing embeddings of difference fields into difference closed fields (as we will see in Theorem~\ref{main:theo3}). It is not hard to see that if a normal extension $\mc{F}$ of the prime field has a splitting algorithm, then $\mc{F}$ has the computable extendability of automorphisms property. Is having a splitting algorithm implied by computable extendability of automorphisms property? Although we don't know if this is true in general (and we conjecture that perhaps not),  we give a condition on the Galois group of $\mc{F}$ over the prime field---the \textit{non-covering property}---under which the computable extendability of automorphisms property is equivalent to having a splitting algorithm.
 \begin{defn}
\label{defn:NCP}
 We say that a group $G$ has the \textit{non-covering property} if for all finite index normal subgroups $M \subsetneq N$ of $G$ and $g \in G$, there is $h \in g N$ such that for all $x \in G$, $x^{-1} h x \notin  g M$.
 \end{defn} 
In Lemma~\ref{forking-equivalence} we will give an equivalent condition in the language of field extensions, using Galois correspondence.  

 Before we state our second main result, we note that groups with the non-covering property include abelian and simple groups, and the class of profinite groups with the non-covering property is closed under direct products.

\begin{thm}\label{maintheo2}
Let $\mc{F}$ be a computable normal extension of $\mathbb{F}_p$, for some prime $p$, such that $\gal(\mc{F}/\mathbb{F}_p)$ has the non-covering property. The following are equivalent:
\begin{enumerate}
	\item $\mc{F}$ has a splitting algorithm,
	\item $\mc{F}$ has the computable extendability of automorphisms property,
	\item $\mc{F}$ has the uniform extendability of automorphisms property.
\end{enumerate}
\end{thm}

 In characteristic $p > 0$, all Galois groups are abelian, and so every Galois group has the non-covering property. Thus, in characteristic $p > 0$ the computable extendability of automorphisms property is equivalent to having a splitting algorithm. 

\subsection{Applications to difference closed fields}

Rabin~\cite{Rabin60} showed that every computable field can be computably embedded into its computable algebraic closure, and Harrington~\cite{Harrington74} later showed that every computable differential field can be computably embedded into a differential closure. We  consider the possibility of such a result for fields with a distinguished automorphism; such structures are called \emph{difference fields}~\cite{ChatzidakisHrushovski99}. An existential closure of such a structure analogous to an algebraically closed field exists and is called a \emph{difference closed field}. (We note that there is no such a thing as \textit{the} difference closure since there might be no ``smallest'' difference closed field containing a given difference field.  The formal definitions will follow later.)
In what follows next, we refer to this hypothetical analogous result as the Rabin-Harrington theorem.

We note that a difference field $(\mc{F}, \sigma)$ may distinguish a rather boring automorphism~$\sigma$, e.g., the identity, for which the Rabin-Harrington theorem clearly holds. On the other hand,  we will see that there exist computable difference fields that do not embed into any computable difference closed field.  Thus, the same field may have two different automorphisms, one witnessing the Rabin-Harrington theorem, and the other witnessing its failure, and finding a satisfactory characterization in this setting seems rather hopeless (yet the reader may try to find one). On the other hand, we are mostly interested in the \emph{properties of the underlying field} which make the Rabin-Harrington theorem hold, and we are not that much concerned with  the properties of some ``pathological'' automorphism that may witness the failure of the Rabin-Harrington theorem.
Thus, we arrive at the third main question addressed in the paper:

\medskip

$\,\,\,\,\,\,\,$ \emph{For which $\mathcal{F}$ does $(\mc{F}, \sigma)$ satisfy the  Rabin-Harrington theorem for all~$\sigma$?}

\medskip
\noindent Here of course $\mc{F}$ is a computable field and $\sigma$ ranges over all computable automorphisms of $\mc{F}$.
We show in Theorem~\ref{thm:embedding-into-acfa} that 
the Rabin-Harrington Theorem holds for difference fields with underlying field $\mc{F}$ if and only if $\mc{F}$ has the computable extension of automorphisms property. 
  Using our results on extending automorphisms, namely  the second main result of the paper (Theorem~\ref{maintheo2}), we can find a large class of difference fields which satisfy the Rabin-Harrington theorem
  for any interpretation of the distinguished automorphism:
  \begin{thm}\label{main:theo3} Let $\mc{F}$ be a computable normal extension of $\mathbb{F}_p$, for some prime $p$, such that $\gal(\mc{F}/\mathbb{F}_p)$ has the non-covering property. Then the following are equivalent:
  \begin{enumerate}
	\item $\mc{F}$ has a splitting algorithm,
	\item for any computable $\sigma$, $(\mc{F}, \sigma)$ can be computably embedded into a computable difference closed field.
	\end{enumerate}
  \end{thm}

\noindent Even without the non-covering property, (1) implies (2).
	
In particular, this theorem gives a complete answer to the third main question of the paper in the case of  a normal extension of $\mathbb{F}_p$ for any $p>0$. On the other hand, Theorem~\ref{main:theo3} will be used to produce various examples of computable difference fields that cannot be embedded into computable difference closed fields. We conclude that the most naive attempt to generalize the results of Rabin  and Harrington \emph{fails}. On the other hand,
if we allow the automorphism to vary, we get a complete characterization for a large class of fields. 
 
 \subsection{The non-covering property} Since our main results refer to the non-covering property of Galois groups, we would like to know more about the class of groups having this property. In Subsection~\ref{subsection:applications} we study the class of groups that have the non-covering property, with an emphasis on profinite groups. It is not hard to see that abelian groups and simple groups have the non-covering property (see Lemma~\ref{forking-equivalence}).
However, it takes a lot more effort to prove:

\begin{thm}\label{prop:product}
Let $\{G_i : i \in I\}$ be a collection of profinite groups, each of which has the non-covering property. Then $\prod_{i \in I} G_i$ has the non-covering property.
\end{thm}

The proof of this theorem might be of some independent interest to the reader. It filters through Goursat's lemma \cite{Goursat89} (to be stated in the proof of Theorem~\ref{prop:product}). We note that our proof uses profiniteness to reduce the case of arbitrarily many direct factors to just two factors, and the proof of the case of just two factors (Lemma~\ref{lemma:keylemma}) does not use profiniteness. We leave open whether one can use profiniteness to simplify our proof of Lemma~\ref{lemma:keylemma}. We also note that some groups do not have the non-covering property (to be discussed). 

 \subsection{The structure of the paper}

We will begin in \S 2 by giving some background on computable fields and difference fields. In \S 3 we will consider embeddings into algebraically closed fields and the computable extendability of embeddings property, and prove the first main result, Theorem \ref{thm:1}. In \S 4 we will  consider automorphisms and the computable extendability of automorphisms property. We begin in \S 4.1 by considering a strengthening of the computable extendability of automorphisms property. In \S 4.2, we prove the second main result, Theorem \ref{maintheo2}. In \S 4.3, we study the class of groups with the non-covering property, and in \S 4.4 we give some applications of Theorem \ref{maintheo2}. In \S 5 we consider applications to difference fields and the Rabin-Harrington theorem. Finally, in \S 6 we state an open problem on the characterization of fields with the computable extendability of automorphisms property.


\section{Preliminaries}\label{Preliminaries}

\subsection{Separable and Purely Inseparable Extensions}

If $\mc{F}$ is a field, a polynomial $f \in \mc{F}[X]$ is called \textit{separable} if it has no repeated roots. A element $a \in \mc{E}$ of an algebraic field extension $\mc{E} / \mc{F}$ is called \textit{separable over} $\mc{F}$ if its minimal polynomial over $\mc{F}$ is a separable polynomial. An algebraic field extension $\mc{E} / \mc{F}$ is called \textit{separable} if every element of $\mc{E}$ is separable over $\mc{F}$. Recall that if $\mc{F}$ is finite or characteristic zero, then it is \textit{perfect}, i.e., every algebraic extension is a separable extension.

An algebraic field extension $\mc{E} / \mc{F}$ is called \textit{purely inseparable} if $\mc{E} \setminus \mc{F}$ contains no separable elements. Equivalently, $\mc{E}$ is a field of characteristic $p > 0$ and every element of $\mc{E}$ is the unique root of a polynomial $X^{p^n} - a = 0$ with $a \in \mc{F}$. Given an algebraic field extension $\mc{E} / \mc{F}$, the set
\[ \mc{F}^s = \{ a \in \mc{E} : a \text{ is separable over } \mc{F}\} \]
is the maximal separable extension of $\mc{F}$ inside of $\mc{E}$ and is called the \textit{separable closure} of $\mc{F}$ in $\mc{E}$. The field extension $\mc{E} / \mc{F}^s$ is purely inseparable. In the special case where $\mc{E} = \overline{\mc{F}}$ is the algebraic closure of $\mc{F}$, $\mc{F}^s$ is called the \textit{separable closure} of $\mc{F}$ and is the maximal separable extension of $\mc{F}$.

An algebraic field extension $\mc{E} / \mc{F}$ is \textit{normal} if every irreducible polynomial in $\mc{F}[X]$ that has a root in $\mc{E}$ factors completely in $\mc{E}[X]$.
A normal separable extension $\mc{E}/\mc{F}$ is called a Galois extension and has associated to it the Galois group $\gal(\mc{E}/\mc{F})$ of automorphisms of $\mc{E}$ fixing $\mc{F}$. Recall that the Galois group obeys the fundamental theorem of Galois theory: the normal subgroups $H \trianglelefteq \gal(\mc{E}/\mc{F})$ correspond to the intermediate normal field extensions.

\subsection{Computable Fields}


Recall that the \textit{splitting set} $S_\mc{F}$ of $\mc{F}$ is the set of all polynomials $p \in \mc{F}[X]$ which are reducible over $\mc{F}$. The splitting set of a field is not necessarily computable (see \cite[Lemma 7]{Miller08}), but it is always c.e. If the splitting set of $\mc{F}$ is computable, then we say that $\mc{F}$ has a \textit{splitting algorithm}. Finite fields and algebraically closed fields trivially have splitting algorithms. Kronecker \cite{Kronecker82} showed that $\mathbb{Q}$ has a splitting algorithm, and also that many other field extensions also have a splitting algorithm:

\begin{thm}[Kronecker \cite{Kronecker82}; see also \cite{vanderWaerden70}]\label{Kronecker}
The field $\mathbb{Q}$ has a splitting algorithm. If a computable field $\mc{F}$ has a splitting algorithm, and $a$ is transcendental over $\mc{F}$, then $\mc{F}(a)$ has a splitting algorithm. If $a$ is separable and algebraic over $\mc{F}$, then $\mc{F}(a)$ has a splitting algorithm. Moreover, the splitting algorithm for $\mc{F}(a)$ is uniform in the minimal polynomial for $a$ over $\mc{F}$.
\end{thm}

Given a field $\mc{F}$ and an element $a$ which is either transcendental over $\mc{F}$, or separable and algebraic over $\mc{F}$, we know that $\mc{F}(a)$ has a splitting algorithm. However, the algorithm depends on whether $a$ is transcendental or algebraic. To find a splitting algorithm uniformly, we must know which is the case.

Rabin \cite{Rabin60} showed that every computable field $\mc{F}$ has a computable algebraic closure $\overline{\mc{F}}$, and moreover there is a computable embedding $\imath \colon \mc{F} \to \overline{\mc{F}}$. We call such an embedding a \textit{Rabin embedding}. Moreover, he characterized the image of $\mc{F}$ under this embedding:

\begin{thm}[Rabin \cite{Rabin60}]
Let $\mc{F}$ be a computable field. Then there is a computable algebraically closed field $\overline{\mc{F}}$ and a computable field embedding $\imath \colon \mc{F} \to \overline{\mc{F}}$ such that $\overline{\mc{F}}$ is algebraic over $\imath(\mc{F})$. Moreover, for any such $\overline{\mc{F}}$ and $\imath$, the image $\imath(\mc{F})$ of $\mc{F}$ in $\overline{\mc{F}}$ is Turing equivalent to the splitting set of $\mc{F}$. 
\end{thm}

A computable field $\mc{F}$ has a \textit{dependence algorithm} if given $a$ and $b_1,\ldots,b_n$, we can compute whether $a$ is algebraically independent over $b_1,\ldots,b_n$. A field has a dependence algorithm if and only if it has a computable transcendence base (see, for example, \cite[Proposition 2.2]{HarrisonTrainorMelnikovMontalban}). In particular, fields of finite transcendence degree have a dependence algorithm.

\medskip

\noindent \emph{Convention.} By an extension $\mc{E} / \mc{F}$ of computable fields, we mean that there is a computable embedding of $\mc{F}$ into $\mc{E}$.

\subsection{Difference fields} 
Difference fields were first studied by Ritt in the 1930s. A good reference on the classical algebraic theory of difference fields is the book by Cohn \cite{Cohn65}. A difference field is a field $\mc{F}$ together with an embedding $\sigma \colon \mc{F} \to \mc{F}$. If $\sigma$ is onto, $(\mc{F},\sigma)$ is called \textit{inversive}. As every difference field has a unique inversive closure up to isomorphism, we lose nothing by assuming that all of our difference fields are inversive.

A difference field $(\mc{F},\sigma)$ is called a \textit{difference closed field} if it is existentially closed in the language of difference fields. Difference closed fields arose in the model theoretic study of difference fields (see \cite{Macintyre97} and \cite{ChatzidakisHrushovski99}). $\mc{F}$ is difference closed if and only if:
\begin{enumerate}
	\item[(i)] $\sigma$ is an automorphism of $\mc{F}$;
	\item[(ii)] $\mc{F}$ is algebraically closed;
	\item[(iii)] For every variety $U$, every affine variety $V \subseteq U \times \sigma(U)$ which projects generically onto $U$ and $\sigma(U)$, and every algebraic set $W \subsetneq V$, there is an $\mc{F}$-rational point $a \in U(\mc{F})$ such that $(a,\sigma(a)) \in V \setminus W$.
\end{enumerate}
The condition (iii) may be viewed as saying that certain systems of equations and inequations have solutions in $\mc{F}$.
Conditions (i), (ii), and (iii) axiomatize the theory $ACFA$ of \textit{difference closed fields}. $ACFA$ is decidable, and moreover the theories $ACFA_p$ of difference closed fields of characteristic $p$ are also decidable for any $p$, including $p = 0$ \cite[(1.4) of][]{ChatzidakisHrushovski99}. $ACFA$ is the model companion of the theory of difference fields \cite[(1.4) of][]{ChatzidakisHrushovski99} and hence every formula is equivalent, modulo $ACFA$, to an existential formula \cite[(1.6) of][]{ChatzidakisHrushovski99}. Thus, we have:

\begin{fact}
Every computable difference closed field has a computable (full) elementary diagram.
\end{fact}

We call a structure with a computable elementary diagram \textit{decidable}; thus every difference closed field is decidable.

\section{Extending Embeddings into the Algebraic Closure}
\label{Embeddings}

We begin by showing that if $\mc{F}$ is any computable field with a splitting algorithm, $\imath \colon \mc{F} \to \overline{\mc{F}}$ is a Rabin embedding, and $\alpha \colon \mc{F} \to \mc{K}$ is a computable embedding of $\mc{F}$ into an algebraically closed field $\mc{K}$, then there is a computable embedding of $\overline{\mc{F}}$ into $\mc{K}$ extending $\alpha$. In particular, the new results here are in the case of characteristic $p > 0$. The new issue we have to deal with in characteristic $p > 0$ is that Theorem \ref{Kronecker} fails for non-separable extensions. We begin by finding the separable closure of a field $\mc{F}$ within its algebraic closure $\overline{\mc{F}}$.

\begin{lem}\label{lem:sep-closure}
Let $\mc{F}$ be a computable field. Then the separable closure of $\mc{F}$ is c.e. If $\mc{F}$ has a splitting algorithm, then the separable closure $\mc{F}^s$ of $\mc{F}$ in $\overline{\mc{F}}$ is computable (so that $\mc{F}^s$ has a splitting algorithm).
\end{lem}
\begin{proof}

Embed $\mc{F}$ in its algebraic closure $\overline{\mc{F}}$. An element $a \in \overline{\mc{F}}$ is separable if and only if there is a polynomial $p(X) \in \mc{F}[X]$ of degree $m$ with $p(a) = 0$ and with $m$ distinct roots in $\overline{\mc{F}}$. Thus the separable closure of $\mc{F}$ is c.e. If $\mc{F}$ has a splitting algorithm, then given $a \in \overline{\mc{F}}$ we can find the minimal polynomial $p$ of $a$ over $\mc{F}$. Then $a$ is separable over $\mc{F}$ if and only if $p$ has no repeated roots, which happens if and only if $p'(a) \neq 0$. (Here, $p'(X)$ is the derivative of $p(X)$ with respect to $X$, treating the coefficients as constants.) So the separable closure of $\mc{F}$ is computable.
\end{proof}

We are now ready to extend an embedding from a field with a splitting algorithm. The main idea is to break the embedding into two steps; first to extend an embedding $\alpha \colon \mc{F} \to \mc{K}$ to an embedding $\beta \colon \mc{F}^s \to \mc{K}$ of the separable closure of $\mc{F}$ into $\mc{K}$, and second to note that $\beta$ extends to a unique embedding of $\overline{\mc{F}}$ into $\mc{K}$ and that this extension is computable from $\beta$.
 
\begin{thm}\label{thm:splitting-implies-turing-functional}
Let $\mc{F}$ be a computable field and $\imath \colon F \to \overline{\mc{F}}$ a Rabin embedding of $\mc{F}$ into its algebraic closure. Suppose that $\mc{F}$ has a splitting algorithm. Then there is a Turing functional $\Phi$ such that whenever $\alpha: \mc{F} \to \mc{K}$ is an embedding of $\mc{F}$ into an algebraically closed field $\mc{K}$, $\Phi^{\alpha \oplus \mc{K}} \colon \overline{\mc{F}} \to \mc{K}$ is an embedding of $\overline{\mc{F}}$ into $\mc{K}$ $\imath$-extending $\alpha$.
\end{thm}
\begin{proof}
Since $\mc{F}$ has a splitting algorithm, the image $\imath(\mc{F})$ of $\mc{F}$ in $\overline{\mc{F}}$ is computable. We may identify $\mc{F}$ with its image. By Lemma \ref{lem:sep-closure} the separable closure $\mc{F}^s$ of $\mc{F}$ is computable as a subset of $\overline{\mc{F}}$ and has a splitting algorithm.

Let $\mc{K}$ be an algebraically closed field and $\alpha \colon \mc{F} \to \mc{K}$ a field embedding. We will begin by describing a procedure to extend $\alpha$ to an embedding $\beta \colon \mc{F}^s \to \mc{K}$. Let $\{a_1,a_2,\ldots\}$ be an enumeration of the elements $F^s$. Start with $\beta$ defined only on $\mc{F}$ and $\imath$-extending $\alpha$. Using the splitting algorithm for $\mc{F}$, find the minimal polynomial $P_1 \in \mc{F}[X]$ of $a_1$ over $\mc{F}$. Find a solution $b_1 \in \mc{K}$ to $\alpha(P_1)$. Then define $\beta$ on $\mc{F}(a_1)$ by mapping $a_1$ to $b_1$. Since $a_1$ is algebraic and separable over $\mc{F}$ (and we know its minimal polynomial), we have a splitting algorithm for $\mc{F}(a_1)$. The separable closure of $\mc{F}(a_1)$ is $\mc{F}^s$. Now find the minimal polynomial $P_2 \in \mc{F}[X]$ of $a_2$ over $\mc{F}(a_1)$, and a solution $b_2$ to $\alpha(P_2)$. Define $\beta$ on $\mc{F}(a_1,a_2)$ by mapping $a_2$ to $b_2$. Note that $a_2$ is separable over $\mc{F}(a_1)$ since
\[ \mc{F} \subseteq \mc{F}(a_1) \subseteq \mc{F}(a_1,a_2) \subseteq \mc{F}^s\]
and $\mc{F}^s$ is a separable algebraic extension of $\mc{F}$. Since $a_2$ is algebraic and separable over $\mc{F}(a_1)$, we have a splitting algorithm for $\mc{F}(a_1,a_2)$. Its separable closure is still $\mc{F}^s$. Continuing in this way, we define an embedding $\beta: \mc{F}^s \to \mc{K}$ which $\imath$-extends $\alpha: \mc{F} \to \mc{K}$.

In characteristic zero, we are done since $\mc{F}^s = \overline{\mc{F}}$. In characteristic $p > 0$, we can extend $\beta$ to an embedding $\overline{\mc{F}} \to \mc{K}$ in the following manner. Given $b \in \overline{\mc{F}}$, find the minimal polynomial $P \in \mc{F}^s[X]$ of $b$ over $\mc{F}^s$ (recalling that $\mc{F}^s$ has a splitting algorithm). Then $P(X)$ is of the form $X^{p^n} - r = 0$ with $r \in \mc{F}$. Note that $b$ is the unique solution of $p(X) = 0$, and we can find the unique solution $c$ to $\beta(p)(X) = 0$. Map $b$ to $c$. This is the unique embedding of $\overline{\mc{F}}$ into $\mc{K}$ extending $\beta$.

The construction was uniform in $\alpha$ and $\mc{K}$, and so we get the desired Turing functional $\Phi$.
\end{proof}

We are now ready to prove Theorem \ref{thm:1}, which says that a field $\mc{F}$ has a splitting algorithm if and only if it has the computable (or uniform) extendability of embeddings property. 

\begin{proof}[Proof of Theorem \ref{thm:1}]
The implication $(1) \Rightarrow (2)$ is Theorem \ref{thm:splitting-implies-turing-functional}. The implication $(2) \Rightarrow (3)$ is immediate. It remains to show the implication $(3) \Rightarrow (1)$.

Fix $\imath: \mc{F} \to \overline{\mc{F}}$, a computable embedding of $\mc{F}$ into a computable presentation $\overline{\mc{F}}$ of its algebraic closure. Suppose that every computable embedding of $\mc{F}$ into a computable algebraically closed field $\mc{K}$ $\imath$-extends to a computable embedding of $\overline{\mc{F}}$ into $\mc{K}$.

We will attempt to construct a computable field $\mc{K}$ and a computable embedding $\alpha \colon \mc{F} \to \mc{K}$ while attempting to diagonalize against all potential computable extensions $\varphi_e : \overline{\mc{F}} \to \mc{K}$ (by having $\alpha(a) \neq \varphi_e(\imath(a))$ for some $a \in \mc{F}$). We know that the construction must fail, and from this we will conclude that $\mc{F}$ has a splitting algorithm.

We construct $\mc{K}$ by an effective Henkin-style construction. The Henkin construction will be similar to one that can be used to prove Rabin's theorem that every field embeds into a computable presentation of its algebraic closure. See, for example, \cite[Theorem 2.5]{FriedmanSimpsonSmith83} where this construction is carried out in reverse mathematics. (Rabin's original proof constructed the algebraic closure using a quotient of a polynomial ring with infinitely many variables.) Let $\mathcal{L}_{F}$ be the language of fields with constant symbols for the elements of $\mc{F}$, and let $T$ be the consistent theory of algebraically closed fields together with the atomic diagram of $\mc{F}$. By quantifier elimination for the theory of algebraically closed fields, $T$ is a complete theory and hence is decidable. We want to construct a decidable prime model of the theory $T$, which gives an algebraic closure $\mc{K}$ of $\mc{F}$ together with an embedding of $\mc{F}$ into $\mc{K}$. The embedding $\alpha: \mc{F} \to \mc{K}$ will be built as part of the Henkin construction. Constructing a prime model requires a slight modification of the Henkin construction, which is possible in this case---we must also omit the type of an element that is transcendental over $\mc{F}$ (see \cite{Millar83} for the general theorem on effectively omitting types).

Let $C = \{c_0,c_1,\ldots\}$ be the new constant symbols for the Henkin construction. The domain of $\mc{K}$ will be the equivalence classes of some computable equivalence relation on $C$. Let $\varphi_e : \overline{\mc{F}} \to C$ be a list of partial computable functions which we interpret as the possible computable embeddings $\overline{\mc{F}} \to \mc{K}$. Let $\{a_0,a_1,a_2,\ldots\}$ be a computable enumeration of $\mc{F}$. We use $\underline{a}_{i}$ to denote the constant symbol associated with $a_i \in \mc{F}$.

\medskip{}

\textit{Construction.} At each stage $s$, we define formulas $\delta_0,\ldots,\delta_s$ in the language $\mathcal{L}_{\mc{F} \cup C}$ which form the partial diagram of $\mc{K}$ at stage $s$. The theory $\Delta = \{\delta_0,\delta_1,\ldots\}$ will be a complete theory extending $T$ which is the complete diagram of the model $\mc{K}$ (with the domain of $\mc{K}$ being the equivalence classes in $C$ by the equivalence relation $c \sim d \Leftrightarrow \Delta \vdash c = d$). At stage $s$, let $\psi_s = \delta_0 \wedge \cdots \wedge \delta_{s-1}$. We can arrange the construction so that the only constant symbols from $\mc{F}$ that appear in $\delta_s$ are $\underline{a}_0,\ldots,\underline{a}_s$. 

At stage $0$, let $\delta_0$ be $c_0 = c_0$. 

At stage $s = 4t + 1$, we try to diagonalize against a $\varphi_e$ for $e \leq t$. Search for an $e \leq t$  and an $i < s + 5$ such that $\varphi_{e,t}(\imath(a_i)) = c_j$ and (where $\bar{c} = (c_0,c_1,\ldots)$ is the sequence of constants from $C$ that appear in $\psi_{s}$): \[ T \nvdash \forall \bar{x} (\psi_{s}[\bar{x}/\bar{c}] \Rightarrow \underline{a}_i = x_j). \] By $\psi_{s}[\bar{x}/\bar{c}]$, we mean that the variables $\bar{x} = (x_0,x_1,\ldots)$ have been substituted for the constants $\bar{c}=(c_0,c_1,\ldots)$. This is a bounded search since $T$ is decidable and we only have to search through finitely many $\underline{a}_i$. If such an $e$ exists, choose the least $e$ such that we have not yet diagonalized against $\varphi_e$. Then set $\delta_s$ to be the formula $\underline{a}_i \neq c_j$ for that $e$. If no such $e$ exists, set $\delta_s$ to be the formula $c_0 = c_0$.

At stages $s = 4t + 2$, $s = 4t + 3$, and $s = 4t + 4$, we act as in the standard method of constructing a computable prime model (e.g., Theorems 5.1 and 7.1 of Harizanov's survey \cite{Harizanov98}), as follows:

At stage $s = 4t + 2$, we add a Henkin witness for $\delta_t$. If $\delta_t$ is of the form $(\exists x) \varphi(x)$, then let $c_i$ be a constant which does not appear in $\psi_s$ and let $\delta_s$ be $\varphi(c_i)$. Otherwise, set $\delta_s$ to be the formula $c_0 = c_0$.

At stage $s = 4t + 3$, we satisfy the completeness requirement for the sentence $\chi_t$ from some fixed listing $(\chi_t)_{t \in \omega}$ of the sentences in the language $\mathcal{L}_{\mathcal{F} \cup C}$. Let $\bar{c}$ be the constants from $C$ which appear in $\psi_s$ and $\chi_t$. Check whether
\[ T \vdash \forall \bar{x} (\psi_s \Rightarrow \chi_t)[\bar{x}/\bar{c}]. \]
If this is the case, let $\delta_s$ be $\chi_t$. Otherwise, let $\delta_s$ be $\neg \chi_t$.

At stage $s = 4t + 4$, we omit the type of an element transcendental over $\mc{F}$. We will have $c_t$ satisfy some polynomial over $\mathcal{F}$. Let $\bar{c}$ be the constants from $C$ which appear in $\psi_s$, except for $c_t$. Search for a polynomial $p(x) \in \mathcal{F}[\overline{X}]$ such that
\[ T \nvdash \forall x \forall \bar{z} (\psi_s[x\bar{z}/c_t\bar{c}] \Rightarrow p(x) \neq 0). \]
Set $\delta_s$ to be the formula $p(c_t) = 0$. Some such polynomial $p$ must exist as the type of a transcendental over $\mathcal{F}$ is a non-principal type.

\medskip{}

\noindent \emph{Verification.} By the standard Henkin construction arguments, we get a decidable prime model $\mc{K}$ whose domain consists of equivalence classes from $C$. We get a computable embedding $\alpha$ of $\mc{F}$ into $\mc{K}$ by mapping $a \in F$ to the element of $\mc{K}$ labeled by the symbol $\underline{a}$. Then $\alpha$ extends to an embedding $\beta$ of $\overline{\mc{F}}$ into $\mc{K}$, which we may represent as a computable map $\varphi_e \colon \overline{\mc{F}} \to C$ (by, say, choosing $\varphi_e(a)$ to be the least element of $C$ in the equivalence class of $\beta(a)$, which we can do computably since the equivalence classes are computable). There is a stage $s_0$ after which we never diagonalize against an $e' < e$. We never diagonalize against $e$.

\begin{claimstar}
Let $b \in \overline{\mc{F}}$ and $t$ be a stage such that $\varphi_{e,t}(b) \downarrow = c_j$ for some $j \in \omega$. Let $s = 4t + 1$. Then $b \in \imath(\mc{F})$ if and only if
there is some $i$ such that \begin{equation} T \vdash \forall \bar{x} (\psi_{s}[\bar{x}/\bar{c}] \Rightarrow \underline{a}_i = x_j).\tag{$*$}\end{equation}
\end{claimstar}
\begin{proof}
Given ($*$), in $\mc{K}$ the constant symbol $a_i$ is interpreted as the equivalence class of $c_j$. Thus $\alpha$ maps $a_i$ to the equivalence class of $c_j$. Since $\beta$ extends $\alpha$ and is one-to-one, $\imath(a_i) = b$.

On the other hand, suppose that $b \in \imath(F)$, say $b = a_i$, and suppose to the contrary that ($*$) does not hold. We have two cases. First, if $i < s + 5$, then we set $\delta_s$ to be the formula $a_i \neq c_j$. Then $\alpha(a_i) \neq c_j = \varphi_e(\imath(a_i))$, which is a contradiction. Second, if $i \geq s + 5$, then let $s' > s$ be the first stage of the form $s' = 4t' + 1$ with $i < s' + 5$. We have $i > s'$ (as if $i \leq s'$ we could have chosen $s' - 4$). Since the only constant symbols from $\mc{F}$ that appear in $\psi_{s'}$ are $\underline{a}_0,\ldots,\underline{a}_{s'}$, and $i > s'$, $\underline{a}_i$ does not appear in $\psi_s$. Then we have
\[T \nvdash \forall \bar{x} (\psi_{s'}[\bar{x}/\bar{c}] \Rightarrow \underline{a}_i = x_j).\]
We set $\delta_{s'}$ to be the formula $\underline{a}_i \neq c_j$ which again yields a contradiction. Hence ($*$) holds.
\end{proof}

The claim gives us a decision procedure for $\imath(F) \subseteq \overline{\mc{F}}$. At any stage $s$, there are only finitely many constants $c \in C$ mentioned in $\psi_s$, and hence only finitely many $a_i$ such that we might possibly have ($*$). So given $b \in \overline{\mc{F}}$, compute $s = 4t+1 \geq s_0$ and $j$ such that $\varphi_{e,t}(b) \downarrow = c_j$, and then check ($*$) for the finitely many possible $a_i$ to decide whether $b \in \imath(\mc{F})$.
\end{proof}

It was important in Theorem \ref{thm:1} that we allow the field $\mc{K}$ to vary. This is because if $\mc{F}$ is a field of infinite transcendence degree, there may be computable algebraically closed fields of infinite transcendence degree into which $\mc{F}$ does not effectively embed. For example, if $\mc{F}$ does not have a dependence algorithm but $\mc{K}$ does, then there is no computable embedding of $\mc{F}$ into $\mc{K}$. If we restrict to the case where $\mc{F}$ is an algebraic field, then $\mc{F}$ has a computable embedding into every computable algebraically closed field $\mc{K}$. In this particular case we get the following corollary, which we use in \S 4, where the field $\mc{K}$ is fixed:

\begin{corollary}\label{cor:algebraic}
Let $\mc{F}$ be a computable algebraic field and $\imath \colon \mc{F} \to \overline{\mc{F}}$ a computable embedding of $\mc{F}$ into a computable presentation of its algebraic closure. Let $\mc{K}$ be a computable algebraically closed field. Then the following are equivalent:
\begin{enumerate}
	\item $\mc{F}$ has a splitting algorithm,
	\item There is a Turing functional $\Phi$ which takes an embedding $\alpha \colon \mc{F} \to \mc{K}$ to an embedding $\Phi^\alpha$ of $\overline{\mc{F}}$ into $\mc{K}$ extending $\alpha$,
	\item Every computable embedding of $\mc{F}$ into $\mc{K}$ $\imath$-extends to a computable embedding of $\overline{\mc{F}}$ into $\mc{K}$.
\end{enumerate}
\end{corollary}
\begin{proof}
By Theorem \ref{thm:1}, it suffices to show that (3) in the statement implies that $\mc{F}$ has the computable extendability property with respect to $\iota \colon \mc{F} \to \overline{\mc{F}}$. Let $\alpha \colon \mc{F} \to \mc{L}$ be a computable embedding of $\mc{F}$ into a computable algebraically closed field $\mc{L}$. We can enumerate, in $\mc{L}$, the algebraic closure of the prime field and this contains the image $\alpha(\mc{F})$ of $\mc{F}$. So we may assume that $\mc{L}$ is the algebraic closure of its prime field.

We can compute an embedding $\jmath\colon \mc{L} \to \mc{K}$ and let $\alpha^* \colon \mc{F} \to \mc{K}$ be $\jmath \circ \alpha$. By (3), there is an embedding $\beta^* \colon \overline{\mc{F}} \to \mc{K}$ which $\imath$-extends $\beta$.
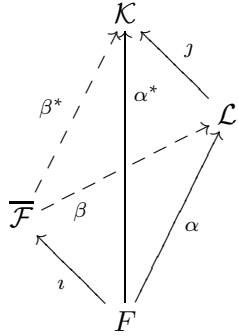
\[\xymatrix{ & \mathcal{K}\\
 &  & \mathcal{L}\ar[ul]_{\jmath}\\
\overline{\mathcal{F}}\ar@{-->}[uur]^{\beta^{*}}\ar@{-->}[urr]_(.25){\beta}|\hole\\
 & F\ar[ul]^{\imath}\ar[uur]_{\alpha}\ar[uuu]_(.75){\alpha^{*}}
}\]
Since $\overline{\mc{F}}$ and $\overline{\mc{L}}$ are both algebraic closures of the prime field, the image of $\beta^*$ is the same as the image of $\jmath$. So there is an embedding $\beta \colon \overline{\mc{F}} \to \mc{L}$ such that $\jmath \circ \beta = \beta^*$. Then $\beta^*$ $\iota$-extends $\alpha$.
\end{proof}

\section{Extending Automorphisms of Normal Extensions of the Prime Field}\label{Automorphisms}

\subsection{Strong extendability of automorphisms property}

In the setting of automorphisms, it is natural to look at normal algebraic extensions of the prime field (see Proposition \ref{fact:rigid}).
When $\mc{F}$ is such an extension, we get the following corollary of Theorem~\ref{thm:1}, with two strengthenings of the computable extendability of automorphisms property.
(We denote the prime field by $\mathbb{F}_p$ even in the case of characteristic $p = 0$.)

\begin{corollary}\label{thm:auto}
Let $\mc{F}$ be a computable normal algebraic extension of the prime field and $\imath: \mc{F} \to \overline{\mc{F}}$ an embedding of $\mc{F}$ into a computable presentation of its algebraic closure. The following are equivalent:
\begin{enumerate}
	\item $\mc{F}$ has a splitting algorithm.
	\item For every computable automorphism $\alpha: \mc{F} \rightarrow \mc{F}$ of $\mc{F}$, there is a computable automorphism $\beta : \overline{\mc{F}} \rightarrow \overline{\mc{F}}$ which $\imath$-extends $\alpha$.
	\item There is a uniform procedure which, given any computable automorphism $\alpha: \mc{F} \rightarrow \mc{F}$ of $\mc{F}$, outputs a computable automorphism $\beta : \overline{\mc{F}} \rightarrow \overline{\mc{F}}$ which $\imath$-extends $\alpha$.
\end{enumerate}
\end{corollary}

\begin{proof}[Proof of Corollary \ref{thm:auto}]
Suppose that $\mc{F}$ is a computable normal algebraic field, and $\imath \colon \mc{F} \to \overline{\mc{F}}$ is a Rabin embedding. If $\mc{F}$ has a splitting algorithm, then by Corollary \ref{cor:algebraic}, any automorphism $\alpha$ of $\mc{F}$ extends to an automorphism of $\overline{\mc{F}}$ (taking $\mc{K} = \overline{\mc{F}}$ in the statement of the corollary). Indeed, $\imath \circ \alpha$ is a computable embedding of $\mc{F}$ into $\overline{\mc{F}}$ and hence there is an automorphism $\beta$ of $\overline{\mc{F}}$ which $\imath$-extends $\imath \circ \alpha$; that is, $\beta \circ \imath = \imath \circ \alpha$. So $\beta$ $\imath$-extends $\alpha$.

On the other hand, suppose that every automorphism of $\mc{F}$ extends to an automorphism of $\overline{\mc{F}}$. We will check (3) of Corollary \ref{cor:algebraic} with $\mc{K} = \overline{\mc{F}}$. Let $\alpha \colon \mc{F} \to \overline{\mc{F}}$ be an embedding. Since $\mc{F}$ is normal, $\alpha(\mc{F}) = \imath(\mc{F})$. Then $\imath^{-1} \circ \alpha \colon \mc{F} \to \mc{F}$ is an automorphism of $\mc{F}$, and hence extends to an automorphism $\beta$ of $\overline{\mc{F}}$. We have the following diagram:
\[\xymatrix{\overline{\mc{F}}\ar[r]^{\beta} & \overline{\mc{F}}\\
\imath(\mc{F})\ar[r]^{\alpha \circ \imath^{-1}}\ar[u]^{\subseteq} & \imath(\mc{F})\ar[u]_{\subseteq}\\
\mc{F}\ar[r]_{\imath^{-1}\circ\alpha}\ar[u]^{\imath}\ar[ur]_\alpha & \mc{F}\ar[u]_{\imath}
}\]
Note that $\beta \colon \overline{\mc{F}} \to \overline{\mc{F}}$ $\imath$-extends the embedding $\alpha$ of $\mc{F}$ into $\overline{\mc{F}}$.
\end{proof}

Note that we had to use Corollary \ref{cor:algebraic} rather than Theorem \ref{thm:1}, because we needed to fix $K = \overline{\mc{F}}$ instead of letting $\mc{K}$ be arbitrary.

In Corollary \ref{thm:auto} we asked for $\mc{F}$ to be a normal extension of $\mathbb{F}_p$. This is required in order to prove the theorem---we will construct an algebraic field which demonstrates that we need the hypothesis of normality in the preceding results. A rigid field automatically satisfies (2) of Corollary \ref{thm:auto}.

\begin{prop}\label{fact:rigid}
There is a rigid computable algebraic field $\mc{F}$ of characteristic zero with no splitting algorithm.
\end{prop}
\begin{proof}
Let $p_0,p_1,\ldots$ list the primes greater than two. Let $\mc{F} = \mathbb{Q}(a_n : n \in \varnothing')$ where $a_n$ is the unique real $p_n$th root of $2$, and $\varnothing'$ is the Turing jump of the empty set. Since $\mc{F} \subseteq \mathbb{R}$, for each $n \in \varnothing'$, $a_n$ is the only $p_n$th root of $2$ in $\mc{F}$. So every automorphism of $\mc{F}$ fixes the $a_n$, and hence fixes $\mc{F}$. Hence $\mc{F}$ is rigid.

We can use an enumeration of $\varnothing'$ to give a computable presentation of $\mc{F}$: $\mc{F}$ can be embedded as a c.e.\ subfield of $\bar{Q}$ and from this we get a computable presentation of $\mc{F}$.

We need to argue that for $n \notin \varnothing'$, $a_n \notin \mc{F}$. We claim that if $n \notin I$, $a_n \notin \mathbb{Q}(a_{i} : i \in I)$. Suppose not; then we can find a finite set $I$ and $n \notin I$ such that $a_n \in \mathbb{Q}(a_{i} : i \in I)$ and for each $j \in I$, $a_j \notin \mathbb{Q}(a_i : i \in I \setminus \{j\})$. Then $\mathbb{Q}(a_i : i \in I)$ is a finite extension of $\mathbb{Q}$ of degree $d = \prod_{i \in I} p_i$. Since $p_n$ does not divide $d$, $\mathbb{Q}(a_n)$ is not a subfield of $\mathbb{Q}(a_i : i \in I)$. This contradicts the assumption that $a_n \in \mathbb{Q}(a_{i} : i \in I)$. Thus for $n \notin \varnothing'$, $a_n \notin \mc{F}$.

No computable presentation of $\mc{F}$ can have a splitting algorithm, as a splitting algorithm would allow us to compute $\varnothing'$.
\end{proof}

\subsection{Computable extendability of automorphisms property}

In Corollary \ref{thm:auto}, we fixed an embedding $\imath \colon \mc{F} \to \overline{\mc{F}}$ and considered only $\imath$-extensions. Now we allow $\imath$ to vary. Note that while every computable presentation of $\overline{\mc{F}}$ is isomorphic, there may be different computable embeddings $\imath: \mc{F} \to \overline{\mc{F}}$ which are not equivalent up to a computable automorphism of $\overline{\mc{F}}$. By Corollary \ref{cor:algebraic}, if $\mc{F}$ is an algebraic field with no splitting algorithm, there are embeddings $\imath$ and $\jmath$ of $\mc{F}$ into $\overline{\mc{F}}$ such that there is no computable automorphism $\sigma$ of $\overline{\mc{F}}$ with $\sigma \circ \imath = \jmath$.

In the introduction, we said that $\mc{F}$ had the \textit{computable extendability of automorphisms property} if each computable automorphism of $\mc{F}$ had such an extension to $\overline{\mc{F}}$. Recall that our interest in the computable extendability of automorphisms property comes from its role in an analogue of Rabin's theorem in the context of difference closed fields; see Theorem \ref{thm:embedding-into-acfa} which we will prove in the following section.

We can already produce examples of fields without the computable extendability of automorphisms property. We use the fact that every noncomputable c.e.\ set is the union of two disjoint, computably inseparable c.e.\ subsets. This is a theorem of Yates, who saw that it followed from a construction of Friedberg; the theorem was subsequently published by Cleave \cite{Cleave70} in 1970.
 
\begin{prop}\label{cor:forking-zero}
For each noncomputable c.e.\ set $C$, the field $\mc{F} = \mathbb{Q}(\sqrt{p_n} : n \in C)$
(with $p_n$ the $n$th prime) does not have the computable extendability of automorphisms property.
\end{prop}
\begin{proof}
Let $A$ and $B$ be disjoint computably inseparable c.e.\ sets with $A \cup B = C$. Recalling the classic result (originally due to Besicovitch \cite{Besicovitch40}) that if $r$ and $q_1,\ldots,q_\ell$ are distinct primes, then $\sqrt{r} \notin \mathbb{Q}(\sqrt{q_1},\ldots,\sqrt{q_\ell})$, we define an automorphism $\alpha$ of $\mc{F}$ by letting $\alpha$ fix the two square roots of $p_n$ if $n\in A$, but interchange them if $n\in B$. This $\alpha$ is computable, but if $\imath$ is any Rabin embedding of $\mc{F}$
into some presentation $\overline{\mc{F}}$ of its algebraic closure and $\beta$ is an automorphism of
$\overline{\mc{F}}$ $\imath$-extending $\alpha$, then $\{n\in\omega~:~\beta(\sqrt{p_n})=\sqrt{p_n}\}$ is a $\beta$-computable separation of $A$ from $B$.
\end{proof}

However, we would like a more complete description of which fields have, and which do not have, the  computable extendability of automorphisms property. We do not obtain a complete description, but we give a characterization in terms of a splitting algorithm for many fields. The idea will be to isolate certain normal extensions $\mc{K} / \mathbb{F}_p$ whose subfield structure behaves sufficiently like the field $\mc{F}$ in Proposition \ref{cor:forking-zero} above, allowing us to make a particular diagonalization argument. In diagonalizing against the computable extendability of automorphisms property, we do not have access to a Rabin embedding $\iota$ (and it does not seem possible to diagonalize against all possible computable Rabin embeddings). So rather than defining $\alpha$ to diagonalize against $\beta$ using the image under a fixed $\iota$, we must define $\alpha$ to diagonalize against \textit{all possible} images under all possible $\iota$. In Proposition \ref{cor:forking-zero}, we do this using the fact that if $\alpha$ fixes the square roots of $p_n$, then so does $\beta$ for any $\iota$-extension of $\alpha$ under any Rabin embedding $\iota$, and similarly if $\alpha$ interchanges the roots of $p_n$. In general, we want to have some finite subfield $\mc{E}$ of $\mc{F}$, and to define $\alpha$ on $\mc{E}$ so that there is no embedding $\iota$ under which $\beta$ $\iota$-extends $\alpha$. We may have already defined $\alpha$ on some subfield of $\mc{E}$, so we do not have a completely free choice of $\alpha$. There are some fields where this argument will always work successfully: those with the non-covering property from Definition \ref{defn:NCP}. In many other fields, we can find some appropriate subfield which satisfies the required condition, allowing the argument to go through---see Examples \ref{ex:one} and \ref{ex:two}.

Using Galois theory, there is also a field-theoretic characterization of the field extensions whose Galois group has the non-covering property, and it is this characterization that we will use in the proof of Theorem \ref{thm:CEAP} (though, in applying the theorem, it will usually be easier to use the group-theoretic characterization). In what follows, it will be helpful to use the language of difference fields to talk about field automorphisms.

\begin{rem}
Let $\mc{F}/\mc{E}$ be a field extension, $\alpha$ an automorphism of $\mc{E}$, and $\beta$ an automorphism of $\mc{F}$. Let $\imath \colon \mc{E} \to \mc{F}$ be a field embedding of $\mc{E}$ into $\mc{F}$. Then $\beta$ $\imath$-extends $\alpha$ if and only if $\imath$ is an embedding of $(\mc{E},\alpha)$ into $(\mc{F},\beta)$ as difference fields.
\end{rem}
\begin{proof}
Both are equivalent to having $\beta \circ \imath = \imath \circ \alpha$.
\end{proof}

\begin{lem}\label{forking-equivalence}
Let $\mc{E} / \mc{F}$ be a separable normal extension. The following are equivalent:
\begin{enumerate}
	\item $\gal(\mc{E}/\mc{F})$ has the non-covering property.
	\item For all finite normal subextensions $\mc{K}_1 / \mc{F}$ and $\mc{K}_2 / \mc{F}$ with $\mc{K}_2 \nsubseteq \mc{K}_1$, and every pair of automorphisms $\sigma$ of $\mc{K}_1$ and $\tau$ of $\mc{K}_2$ fixing $\mc{F}$, there is an automorphism $\alpha$ of $\mc{E}$ extending $\sigma$ and incompatible with $\tau$ (i.e., $(\mc{K}_2,\tau)$ does not embed into $(\mc{E},\alpha)$ as a difference field).
\end{enumerate}
\end{lem}

The second point is related to the monadic and incompatible extensions of difference fields studied by Cohn~\cite{Cohn52}, Babbitt~\cite{Babbitt62}, and Evanovich~\cite{Evanovich73}.

\begin{proof}
We begin by showing (1)$\Rightarrow$(2). Let $\mc{K}_1$ and $\mc{K_2}$ be as in (2). Let $\sigma$ and $\tau$ be automorphisms of $\mc{K}_1$ and $\mc{K}_2$ respectively fixing $\mc{F}$. Let $G = \gal(\mc{E}/\mc{F})$. Let $M$ be the normal subgroup of automorphisms fixing $\mc{K}_2$, and $N$ the normal subgroup of automorphisms fixing $\mc{K}_1$. Since $\mc{K}_1$ and $\mc{K}_2$ are finite extensions, $M$ and $N$ are of finite index. We also have $N \nsubseteq M$. Let $g_1 \in G$ be an automorphism of $\mc{E}$ extending $\sigma$, and $g_2$ an automorphism of $\mc{E}$ extending $\tau$.

We will argue that there is an $h \in g_1 N$ such that for all $z \in G$, $z^{-1} h z \notin g_2 M$. Such an $h$ is an automorphism of $\mc{E}$ extending $\sigma$, and $g_2 M$ is the set of automorphisms of $\mc{E}$ extending $\tau$. Since, for all $x \in G$, $x^{-1} g_1 h x \notin g_2 M$, $(\mc{E},\alpha)$ is not isomorphic as a difference field to $(\mc{E},\beta)$ for any extension $\beta$ of $\tau$.

First, we argue that it suffices to assume $M \subsetneq N$. Suppose that there is $h' \in g_1 NM$ such that for all $z \in G$, $z^{-1} h z \notin g_2 M$. Then write $h' = g_1 n m$. Suppose for some $z \in G$ that $z^{-1} g_1 n z \in g_2 M$, say $z^{-1} g_1 n z = g_2 m'$ with $m' \in M$. Let $m'' \in M$ be such that $z^{-1} m z = m''$. Then
\[ z^{-1} g_1 n m z = g_2 m' m''^{-1} \in g_2 M. \]
This contradicts the choice of $h' = g_1 n m$. So for all $z \in G$, $z^{-1} g_1 n z \notin g_2 M$. Then $h = g_1 n \in g_1 N$ is the automorphism of $\mc{E}$ that we desire. So we may replace $N$ by $NM$.

Now we have two cases. First, suppose that there is no $z \in G$ such that $z^{-1} g_1 z \in g_2 M$. Then $h = g_1$ is as desired.

Second, suppose that for some $c \in M$ and $z \in G$, $z^{-1} g_1 z = g_2 c$. Then $g_1 = z g_2 c z^{-1} = z g_2 z^{-1} c'$ for some other $c' \in M$ since $M$ is a normal subgroup. So $z g_2 z^{-1} = g_1 m$, where $m = (c')^{-1}$. Using the fact that $G$ has the non-covering property, choose $h \in N$ such that for all $x \in G$, $x^{-1} g_1 h x \notin g_1 M$. We claim that for all $x \in G$, $x^{-1} g_1 h x \notin g_2 M$. Suppose to the contrary that there is $x \in G$ such that $x^{-1} g_1 h x \in g_2 M$. Since $x^{-1} g_1 h x \in g_2 M$ and $M$ is a normal subgroup, $g_1 h \in x g_2 x^{-1} M$. We have
\[ x g_2 x^{-1} = (x z^{-1}) z g_2 z^{-1} (x z^{-1})^{-1} = (x z^{-1}) g_1 m (x z^{-1})^{-1}. \]
Let $y = (x z^{-1})$. Since $m \in M$ is a normal subgroup, $y g_1 m y^{-1} = y g_1 y^{-1} m'$ for some other $m' \in M$. Thus $g_1 h \in y g_1 y^{-1} M$ and so $y^{-1} g_1 h y \in g_1 M$. This contradicts the choice of $h$. So for all $x \in G$, $x^{-1} g_1 h x \notin g_2 M$.

The direction (2)$\Rightarrow$(1) proceeds simply by the Galois correspondence. Fix finite index normal subgroups $M \subsetneq N$ of $G = \gal(\mc{E}/\mc{F})$ and $g \in G$. Let $\mc{K}_1$ and $\mc{K}_2$ be the fields fixed by $N$ and $M$ respectively; we have $\mc{K}_1 \subsetneq \mc{K}_2$. The $\sigma$ be the restriction of $g$ to $\mc{K}_1$ and $\tau$ its restriction to $\mc{K}_2$. There is an automorphism $\alpha$ of $\mc{E}$ extending $\sigma$ and not compatible with $\tau$. So $\alpha \in gN$ and for all extensions $\beta$ of $\tau$ to $\mc{E}$ (i.e., $\beta \in g M$), $(\mc{E},\alpha)$ is not isomorphic as a difference field to $(\mc{E},\beta)$. That is, for all $\gamma \in G$ and $\beta \in g M$, $\gamma \circ \alpha \neq \beta \circ \gamma$.
\end{proof}

We will restrict our attention to field extensions whose Galois group has the non-covering property, but we will allow the base field to be an extension of $\mathbb{F}_p$ with a splitting algorithm. The main theorem of this section is as follows. (We state it in a slightly more general form than it appears in the introduction.)

\begin{thm}\label{thm:CEAP}
Let $\mc{E}$ be a computable normal extension of $\mathbb{F}_p$ and let $\mc{F} \subseteq \mc{E}$ be a subfield of $\mc{E}$ with a splitting algorithm which is also a normal extension of $\mathbb{F}_p$. Suppose that $\gal(\mc{E}/\mc{F})$ has the non-covering property. The following are equivalent:
\begin{enumerate}
	\item $\mc{E}$ has a splitting algorithm,
	\item $\mc{E}$ has the computable extendability of automorphisms property,
	\item $\mc{E}$ has the uniform extendability of automorphisms property.
\end{enumerate}
\end{thm}

Many applications of this theorem will have $\mc{F} = \mathbb{F}_p$, but the freedom to choose $\mc{F}$ will allow us to apply the theorem in situations where $\gal(\mc{E}/\mathbb{F}_p)$ does not have the non-covering property.  Producing an example where the theorem cannot be applied seems to be  a non-trivial task, and we do not know of any such examples. See \S \ref{subsection:applications} for some applications of the theorem.

\begin{proof}[Proof of Theorem \ref{thm:CEAP}]
We already know that the implications $(1) \Rightarrow (2)$ and $(1) \Rightarrow (3)$ are true even given a fixed embedding of $\mc{E}$ into $\overline{\mc{E}}$. $(3)$ clearly implies $(2)$. We must show $(2) \Rightarrow (1)$.

Suppose that every computable automorphism of $\mc{E}$ extends to a computable automorphism of $\overline{\mc{E}}$ (via some embedding of $\mc{E}$ into $\overline{\mc{E}}$). We will attempt to construct a computable automorphism $\alpha \in \gal(\mc{E}/\mc{F})$ while diagonalizing against possible computable automorphisms $\varphi_e \colon \overline{\mc{E}} \to \overline{\mc{E}}$ by making sure that the difference field $(\mc{E},\alpha)$ does not embed into the difference field $(\overline{\mc{E}},\varphi_e)$. It suffices to ensure that some difference subfield of $(\mc{E},\alpha)$ does not embed into $(\overline{\mc{E}},\varphi_e)$. We know that the construction must fail, and from this we will conclude that $\mc{E}$ has a splitting algorithm.

Note that the field $\mc{F}$ has a splitting algorithm and is perfect (since it is an algebraic extension of a perfect field), so any finite algebraic extension of $\mc{F}$ has a splitting algorithm which we can determine effectively from a generating set for the extension.

We will require a special enumeration $\{a_1,a_2,\ldots\}$ of $\mc{E}$ with the following properties:
\begin{enumerate}
	\item[(i)] for each $n$, $\mc{F}(a_1,\ldots,a_n)$ is a normal extension of $\mathbb{F}_p$, and
	\item[(ii)] for each $n$, there are no normal extensions of $\mathbb{F}_p$ which are strictly contained between $\mc{F}(a_1,\ldots,a_n)$ and $\mc{F}(a_1,\ldots,a_n,a_{n+1})$.
\end{enumerate}
We can find such an enumeration using the primitive element theorem and Galois theory, as follows. Suppose that we have already defined $a_1,\ldots,a_n$. Given a new element $x$ of $\mc{E}$, first check whether $x \in \mc{F}(a_1,\ldots,a_n)$ using the splitting algorithm for this field. If $x$ is in $\mc{F}(a_1,\ldots,a_n)$, we can safely set $a_{n+1} = x$. Otherwise, compute the conjugates $x = x_1,\ldots,x_\ell$ of $x$ over $\mathbb{F}_p$. Search for a single element $y$ such that \[\mc{F}(y) \in \mc{F}(a_1,\ldots,a_n,x_1,\ldots,x_\ell).\] Such an element exists by the primitive element theorem as $\mc{F}(a_1,\ldots,a_n,x_1,\ldots,x_\ell)$ is a finite separable extension of $\mc{F}$. Now we can compute the Galois group $\gal(\mc{F}(y)/\mc{F})$ as each automorphism of $\mc{F}(y)$ is determined by where it maps $y$. We can compute the normal subgroups and hence the normal extensions of $\mc{F}$ contained between $\mc{F}(a_1,\ldots,a_n)$ and $\mc{F}(a_1,\ldots,a_n,a_{n+1})$. Let
\[ \mc{F}(a_1,\ldots,a_n) \subsetneq \mc{K}_1 \subsetneq \cdots \subsetneq \mc{K}_m = \mc{F}(a_1,\ldots,a_n,a_{n+1})\]
be a maximal chain of normal extensions of $\mathbb{F}_p$. We can compute for each $\mc{K}_i$ a primitive generator over $\mc{F}$ and add these to the enumeration in order (with $y$ chosen as the primitive generator of $\mc{K}_m = \mc{F}(a_1,\ldots,a_n,a_{n+1})$).

\vspace*{10pt}
\textit{Construction.}
At each stage $s$, we will have defined an embedding $\alpha_s\colon \mc{F}\{a_1,\ldots,a_s\} \to \mc{E}$ fixing $\mc{F}$ such that $\alpha_0 \subseteq \alpha_1 \subseteq \cdots \subseteq \alpha_s$. Begin with $\alpha_0 \colon \mathbb{F}_p \to \mc{F}$.

At stage $s+1$, we are given $\alpha_s$. Use the splitting algorithm for $\mc{F}\{a_1,\ldots,a_s\}$ to check whether $a_{s+1} \in \mc{F}\{a_1,\ldots,a_s\}$. If it is, set $\alpha_{s+1} = \alpha_s$. Otherwise, check whether there is $e \leq s$ against which we have not yet diagonalized such that
\begin{enumerate}
	\item $a_i \in \mc{F}(a_1,\ldots,a_s,a_{s+1}) \setminus \mc{F}(a_1,\ldots,a_s)$ and
	\item for each $x \in \overline{\mc{E}}$ which satisfies the same minimal polynomial over $\mc{F}$ as $a_i$, $\varphi_{e,s}(x) \downarrow = c$ for some $c \in \overline{\mc{E}}$.
\end{enumerate}
This is a computable search. We have splitting algorithms for $\mc{F}(a_1,\ldots,a_s,a_{s+1})$ and $\mc{F}(a_1,\ldots,a_s)$, so we can check (1) for a given $a_i$. Also, $\varphi_{e,s}(x)$ converges only for $x$ among the first $s$-many elements of $\overline{\mc{E}}$, and we can use our splitting algorithms to compute the finite set of $a_i$ satisfying (1) and also satisfying the same minimal polynomial as some such $x$.

If there is such an $e$, choose the least one. Let $x_1,\ldots,x_n$ be the conjugates of $a_i$ over $\mc{F}$. By property (ii) of the enumeration,
\[ \mc{F}(a_1,\ldots,a_s,a_{s+1}) = \mc{F}(a_1,\ldots,a_s,x_1,\ldots,x_n).\]
Now we can extend $\varphi_{e,s}$ in a unique way to a computable automorphism of $\mc{F}(x_1,\ldots,x_n)$. If this automorphism is not the identity on $\mc{F}$, then since $\mc{F}$ is normal, $\varphi_{e,s}$ will be incompatible with $\alpha$ no matter how we define $\alpha$. Suppose that $\varphi_{e,s}$ is the identity on $\mc{F}$. Since $\gal(\mc{E}/\mc{F})$ has the non-covering property, by Lemma \ref{forking-equivalence} we can extend $\alpha_s$ to an automorphism of $\alpha_{s+1}$ of $\mc{F}(a_1,\ldots,a_s,a_{s+1})$ which is incompatible with the automorphism $\varphi_{e,s}$ on $\mc{F}(x_1,\ldots,x_n)$, in the sense that $(\mc{F}(x_1,\ldots,x_n),\varphi_e)$ does not embed as a difference field into $(\mc{F}(a_1,\ldots,a_s,a_{s+1}),\alpha_{s+1})$. We can do all of this computably by looking at the actions of the automorphisms on the generators of the fields.

\vspace*{10pt}

\noindent \emph{Verification.} We get an automorphism $\alpha = \bigcup_s \alpha_s$ of $\mc{E}$ which fixes $\mc{F}$. Now we know that for some $e$, $\varphi_e$ is an automorphism of $\overline{\mc{E}}$ such that $(\mc{E},\alpha)$ embeds into $(\overline{\mc{E}},\varphi_e)$ as a difference field. We claim that $\mc{E}$ has a splitting algorithm. The proof will be to show that we can compute the image of $\mc{E}$ in $\overline{\mc{E}}$ (since $\mc{E}$ is a normal extension of $\mathbb{F}_p$, this image is unique; we may fix some embedding $\imath \colon \mc{E} \to \overline{\mc{E}}$ and show that the image of $\mc{E}$ under $\imath$ is computable in $\overline{\mc{E}}$).

Let $s$ be a stage after which we never diagonalize against an $e' \leq e$. Fix $x \in \overline{\mc{E}}$, and let $x = x_1,x_2,\ldots,x_n$ be the conjugates of $x$ over $\mc{F}$. Let $t \geq s$ be a stage by which $\varphi_e(x_i)$ has converged for each $i$. Since $\mc{F}(a_1,\ldots,a_t)$ has a splitting algorithm, we can compute its image $\imath(\mc{F}(a_1,\ldots,a_t))$ in $\overline{\mc{E}}$.

\begin{claimstar}
$x \in \imath(\mc{E})$ if and only if $x \in \imath(\mc{F}(a_1,\ldots,a_t))$.
\end{claimstar}
\begin{proof}

If $x \in \imath(\mc{F}(a_1,\ldots,a_t))$ then $x \in \imath(\mc{E})$. On the other hand, suppose that $x \in \imath(\mc{E})$, say $x = \imath(a_i)$, and suppose to the contrary that $a_i \notin \mc{F}(a_1,\ldots,a_t)$. Now, for some $t' > t$, we have
\[ a_i \in \mc{F}(a_1,\ldots,a_{t'+1}) \setminus \mc{F}(a_1,\ldots,a_{t'}). \]
Then at stage $t'+1$, we define $\alpha_{t'+1} \subset \alpha$ such that $(\mc{F}(x_1,\ldots,x_n),\varphi_e)$ does not embed into $(\mc{F}(a_1,\ldots,a_{t'+1}),\alpha_{t'+1})$ as a difference field. Since $\mc{F}(x_1,\ldots,x_n)$ and $\mc{F}(a_1,\ldots,a_{t'+1})$ are both normal extensions of $\mathbb{F}_p$ (with the former contained in the latter), $(\mc{E},\alpha)$ cannot embed into $(\overline{\mc{E}}),\varphi_e)$.
\end{proof}

From the claim we get a decision procedure for $\imath(\mc{E})$. Given $x \in \overline{\mc{E}}$, compute a stage $t \geq s$ at which $\varphi_e$ converges when applied to all of the conjugates of $x$ over $\mathbb{F}_p$. Using the splitting algorithm for $\mc{F}(a_1,\ldots,a_t)$, we check whether $x \in \imath(\mc{F}(a_1,\ldots,a_t))$ and hence whether $x \in \imath(\mc{E})$. 
\end{proof}


\subsection{The non-covering property}

To apply Theorem \ref{thm:CEAP}, we need a field extension whose Galois group has the non-covering property. We now give some examples of groups with the non-covering property before giving an example of an application of Theorem \ref{thm:CEAP}.

\begin{lem}\label{lem:abelian}
The following groups have the non-covering property:
\begin{enumerate}
	\item abelian groups,
	\item simple groups,
	\item the quaternion group.
\end{enumerate}
\end{lem}
\begin{proof}
(1) Let $G$ be an abelian group. Let $M \subsetneq N$ be normal subgroups of finite index, and fix $g \in G$. Let $h$ be an element of $g(N \setminus M)$. Then for all $x \in G$, $x^{-1} h x = h \notin gM$.

(2) Let $G$ be a simple group. Let $M \subsetneq N$ be normal subgroups of finite index, and fix $g \in G$. Then $N = G$ and $M$ is the trivial subgroup. Then $gN = G$; if $g = e$, pick $h \neq e$, and otherwise pick $h = e$. Then $gM = \{g\}$ and $h$ and $g$ are in different conjugacy classes.

(3) Let $G = \{\pm 1, \pm i, \pm j, \pm k\}$ be the quaternion group. The normal subgroups are $\{1\}$, $\{1,-1\}$, $\{1,-1,i,-i\}$, $\{1,-1,j,-j\}$, $\{1,-1,k,-k\}$, and $G$. The conjugacy classes are $\{1\}$, $\{-1\}$, $\{i,-i\}$, $\{j,-j\}$, and $\{k,-k\}$. It is easy to see that every coset is a disjoint union of conjugacy classes. Thus, given normal subgroups $M \subsetneq N$ and $g \in G$, there is a conjugacy class in $g N$ which is not in $g M$; let $h$ be in this conjugacy class.
\end{proof}

The example from Proposition \ref{cor:forking-zero} has an abelian Galois group $\prod_{n \in \omega} C_2$, and hence Proposition \ref{cor:forking-zero} follows immediately from Theorem \ref{thm:CEAP}. Also, in characteristic $p > 0$ we have:

\begin{thm}\label{cor:forking-p}
Let $\mc{E}$ be a computable normal extension of $\mathbb{F}_p$ in characteristic $p > 0$. The following are equivalent:
\begin{enumerate}
	\item $\mc{E}$ has a splitting algorithm,
	\item $\mc{E}$ has the computable extendability of automorphisms property,
	\item $\mc{E}$ has the uniform extendability of automorphisms property.
\end{enumerate}
\end{thm}
\begin{proof}
The Galois group of every normal extension $\mc{K} / \mathbb{F}_p$ in characteristic $p > 0$ is abelian and hence has the non-covering property. Theorem \ref{thm:CEAP} finishes the proof.
\end{proof}

We can also take arbitrary products of Galois groups with the non-covering property and produce another group with the non-covering property. We must assume that the groups are profinite, but as every Galois group is profinite, this is not a restriction. See \cite{FriedJarden08} for an introduction to profinite groups.

\medskip

\noindent {\bf Theorem~\ref{prop:product}.} \textit{Let $\{G_i : i \in I\}$ be a collection of profinite groups, each of which has the non-covering property. Then $\prod_{i \in I} G_i$ has the non-covering property.}

\medskip

\begin{proof}
We reduce to the case of a product of two groups. If $M \subsetneq N$ are normal subgroups of $\prod_{i \in I} G_i$ of finite index, then $M$ contains a finite intersection of the groups
\[ \hat{G}_i = \{ (x_j)_{j \in I} : x_i = e \}. \]
The intersection of all of the $\hat{G}_i$ is the trivial group, so $\bigcap \hat{G}_i \subseteq M$. Moreover, it is easy to check that these groups are open in the profinite topology of the profinite group $\prod_{i \in I} G_i$ (which is just the product topology) and hence they are closed as well. As the profinite topology is compact, $M$ contains $\hat{G}_{i_1} \cap \cdots \cap \hat{G}_{i_n}$ for some $i_1,\ldots,i_n$.

Let $M',N' \subseteq G_{i_1} \times \cdots \times G_{i_2}$ be the projection of $M$ and $N$ to these indices; $M'$ and $N'$ are normal subgroups. Then
\[ (\prod_{i \in I} G_i) / M \cong (G_{i_1} \times \cdots \times G_{i_n}) / M'.\]
We will prove in  the following lemma that $G_{i_1} \times \cdots \times G_{i_n}$ has the non-covering property, and we can use this to check (for $M$ and $N$) that $\prod_{i \in I} G_i$ has the non-covering property.

\begin{lem}\label{lemma:keylemma}
Let $G$ and $H$ be groups which both have the non-covering property. Then $G \times H$ has the non-covering property.
\end{lem}
\begin{proof}
Let $M \subsetneq N$ be normal subgroups of $G \times H$. Let $\pi_1$ and $\pi_2$ be the projections onto $G$ and onto $H$ respectively.

\begin{case} We have $\pi_1(M) \subsetneq \pi_1(N)$.\end{case}
\noindent Let $a = (a_1,a_2) \in G \times H$ and $b = (b_1,b_2) \in G \times H$ be arbitrary. Choose $g = a_1 g' \in a_1 \pi_1(N)$ such that for all $x \in G$, $x^{-1} g x \notin b_1 \pi_1(M)$. Let $h' \in H$ be such that $(g',h') \in N$, and let $h = a_2 h'$. Then $f = (g,h) \in a N$ is such that for all $z = (x,y) \in G \times H$, $z^{-1} f z \notin b M$.

\begin{case}We have $\pi_2(M) \subsetneq \pi_2(N)$.\end{case}
\noindent Similar to Case 1.

\begin{case}$\pi_1(M) = \pi_1(N)$ and $\pi_2(M) = \pi_2(N)$.\end{case}
\noindent Define $M_1 \subseteq G$ and $M_2 \subseteq H$ by
\[ M_1 = \{ x \in G : (x,e) \in M \} \text{ and } M_2 = \{ y \in H : (e,y) \in M \}. \]
Then $M_1 \times M_2 \subseteq M$. Define $N_1$ and $N_2$ similarly. We have $M_1 \subseteq N_1$ and $M_2 \subseteq N_2$.

\begin{claim}$M_1$ and $N_1$ are normal subgroups of $G$ and $M_2$ and $N_2$ are normal subgroups of $H$.
\end{claim}
\begin{proof}
We show that $M_1$ is a normal subgroup of $G$. Let $m \in M_1$ and $x \in G$. Let $x' = (x,e)$ and $m' = (m,e)$. Then, since $M$ is a normal subgroup of $G$, $(x^{-1} m x,e) = x'^{-1} m' x' \in M$. Hence $x^{-1} m x \in M_1$.
\end{proof}

\begin{claim}$M_1 \subsetneq N_1$ and $M_2 \subsetneq N_2$.
\end{claim}
\begin{proof}
We use Goursat's lemma:
 \begin{lemstar}[\cite{Goursat89}] Let $G_1$ and $G_2$ be groups. Let $H$ be a subgroup of $G_1 \times G_2$ such that the projections $\pi_1\colon H \to G_1$ and $\pi_2\colon H \to G_2$ are surjective. Let $N_1$ and $N_2$ be the kernels of $\pi_2$ and $\pi_1$ respectively; $N_1$ can be identified as a normal subgroup of $G_1$, and $N_2$ as a normal subgroup of $G_2$. Then the image of $H$ in $G_1 / N_1 \times G_2 / N_2$ is isomorphic to the graph of an isomorphism between $G_1 / N_1$ and $G_2 / N_2$. \end{lemstar}
 By Goursat's lemma,
 the image of $M$ in $\pi_1(M) / M_1 \times \pi_2(M) / M_2$ is the graph of an isomorphism $\pi_1(M) / M_1 \cong \pi_2(M) / M_2$. The same is true with $M$ replaced by $N$. Since $\pi_1(M) = \pi_1(N)$, $\pi_2(M) = \pi_2(N)$, and $M \subsetneq N$, we must have $M_1 \subsetneq N_1$ and $M_2 \subsetneq N_2$.
\end{proof}

\begin{claim}
$[G,\pi_1(M)] \subseteq M_1$ and $[H,\pi_2(M)] \subseteq M_2$. Thus $[G \times H,\pi_1(M) \times \pi_2(M)] \subseteq M_1 \times M_2$.
\end{claim}
\begin{proof}
Let $g \in G$ and $m \in \pi_1(M)$. Let $g' = (g,e)$ and $m' = (m,e)$. Since $M$ is a normal subgroup of $G \times H$, $[g',m'] = ([g,m],e) \in M$. Thus $[g,m] \in M_1$.
\end{proof}

Fix $g \in G \times H$ for which we will show that there is $h \in g N$ such that for all $x \in G \times H$, $x^{-1} h x \notin g M$. This will finish the proof of the proposition. Since $N_2 \supsetneq M_2$, we can choose $b \in \pi_2(g) M_2$ such that for all $y \in H$, $y^{-1} b y \notin \pi_2(g) M_2$. Choose $a = \pi_1(g)$. Then $(a,b) \in g N$. Suppose that $(x,y) \in G \times H$ is such that $(x^{-1} a x, y^{-1} b y) \in g M$. Let $m \in M$ be such that $x^{-1} a x = \pi_1(g m)$. 

\begin{claim}
$\pi_1(m) \in M_1$.
\end{claim}
\begin{proof}
Suppose to the contrary that $\pi_1(m) \notin M_1$. Let $m_1 = \pi_1(m)$ and $g_1 = a = \pi_1(g)$. We have $x^{-1} g_1 x = g_1 m_1$. Let $K$ be the subgroup of $G$ generated by $M_1$ and $m_1$. Since $M_1$ is a normal subgroup of $G$, each element of $K$ can be written in the form $k m_1^\ell$ for some $k \in M_1$ and $\ell \in \mathbb{N}$. $K$ is a normal subgroup of $G$ since $[G,m_1] \in M_1$. If $m_1 \notin M_1$, then $M_1$ is a proper subgroup of $K$. So there is $h \in K$ such that for all $z \in G$, $z^{-1} g_1 h z \notin g_1 M_1$. Let $r$ be such that $m_1^r = e$ and let $h = k m_1^\ell$ with $k \in M_1$ and $\ell < r$. Then since $[x,m_1] \in M_1$,
\[ x^{-(r - \ell)} g_1 h x^{r - \ell} \in x^{-(r - \ell)} g_1 x^{r - \ell} m_1^\ell M_1 = g_1 m_1^r M_1 = g_1 M_1. \]
This is a contradiction which proves the claim.
\end{proof}

Since $\pi_1(m) \in M_1$, we have $(e,\pi_2(m)) = m - (\pi_1(m),e) \in M$, and so $\pi_2(m) \in M_2$. But $y^{-1} b y = \pi_2(g m) \notin \pi_2(g) M_2$, a contradiction. This completes the proof of the lemma.
\end{proof}\renewcommand{\qedsymbol}{}
\end{proof}

\subsection{Examples}
\label{subsection:applications}

We can apply Theorem \ref{prop:product} to construct groups having the non-covering property from the groups in Lemma~\ref{lem:abelian}. In all cases, we know that if the field $\mc{E}$ has a splitting algorithm, then it has the computable extendability of automorphisms property.

We begin by noting that there exist groups without the non-covering property:

\begin{prop}
The following groups do not have the non-covering property: $S_3$, $D_8$, and $A_4$.
\end{prop}
\begin{proof}
For $S_3$, let $M = \{e\}$ and $N$ the normal subgroup of rotations. Let $g$ be a reflection. Then $gN$ is the set of all reflections, and all reflections are conjugate.

Write $D(8) = \{e,a,a^2,a^3,x,ax,a^2x,a^3x\}$. Let $M = \{e\}$, $N = \{e,a^2\}$, and $g = a$. Then $aM = \{a\}$ and $aN = \{a,a^3\}$. We have $x^{-1} a x = a^3$.

For $A_4$, let $M = \{e\}$ and $N$ the normal subgroup of $A_4$ isomorphic to $C_2 \times C_2$. Let $g$ be the permutation $(1,2,3)$. Then $gN$ consists of $(1,2,3)$, $(1,4,2)$, $(2,4,3)$, and $(1,3,4)$ all of which are conjugate.  
\end{proof}

Even if $\gal(\mc{E}/\mathbb{F}_p)$ does not have the non-covering property, we can still sometimes apply Theorem \ref{thm:CEAP} either by finding the right field $\mc{F}$ as in the statement of the theorem, or using Lemma \ref{lem:subfield} below with a subfield $\mc{F}$ and applying Theorem \ref{thm:CEAP} to the field extension $\mc{F}/\mathbb{F}_p$. The following two examples illustrate these methods. We begin with a field extension $\mc{E} / \mathbb{Q}$ whose Galois group does not have the non-covering property, but we can use the freedom in choosing the field $\mc{F}$ in the statement of Theorem \ref{thm:CEAP} to apply the theorem.

\begin{exam}\label{ex:one}
Let $\mc{E} = \mathbb{Q}(\omega,\sqrt[3]{p_n} : n \in \varnothing')$ where $\omega$ is a primitive cube root of unity. Note that $\gal(\mc{E} / \mathbb{Q})$ does not have a forking lattice of subgroups for the same reason as $S_3$, because its Galois group is
\[ \gal(\mc{E} / \mathbb{Q}) = \prod_{i \in \omega} C_3 \rtimes C_2 \]
with $C_2$ acting on $C_3$ by inverting elements. Here, we need to know that the intersection of the fields $\mathbb{Q}(\omega,\sqrt[3]{p_n} : n \in U)$ and $\mathbb{Q}(\omega,\sqrt[3]{p_n} : n \in V)$ for $U$ and $V$ disjoint is the field $\mathbb{Q}(\omega)$. See \cite{Mordell53}.

Let $\mc{F} = \mathbb{Q}(\omega)$. Then $\mc{F}$ has a splitting algorithm. $\gal(\mc{E}/\mc{F}) = \prod_{i \in \omega} C_3$ which is abelian. Since $\mc{E}$ does not have a splitting algorithm, by Theorem \ref{thm:CEAP} it does not have the computable extension of automorphisms property.
\end{exam}

The following lemma will allow us to consider a subextension of $\mc{E}$; this will be useful when the Galois group of the extension does not have the non-covering property, but it has a quotient which does.

\begin{lem}\label{lem:subfield}
Let $\mc{E} \supseteq \mc{F} \supseteq \mathbb{F}_p$ be computable algebraic extensions such that $\mc{E}$ is a normal extension of $\mathbb{F}_p$. Suppose that given $x \in \mc{E}$, we can compute the minimal polynomial of $x$ over $\mc{F}$. Then if $\mc{E}$ has the computable extendability of automorphisms property, $\mc{F}$ does as well.
\end{lem}
\begin{proof}
This follows from the fact that we can computably extend an automorphism of $\mc{F}$ to an automorphism of $\mc{E}$ in the style of Theorem \ref{thm:splitting-implies-turing-functional} and uses the fact that $\mc{F}$ is a perfect field. 
\end{proof}

We now have an example where we apply this lemma together with Theorem \ref{thm:CEAP}.

\begin{exam}\label{ex:two}
This example is quite complicated. The idea is to product a field extension whose Galois group is $\prod_{n \in \omega} S_3$, but which does not have a splitting algorithm.

Let $q_0,q_1,\ldots$ be a list of infinitely many distinct primes in the arithmetic progression $4n + 27$, and let $a_n$ be such that $4a_n + 27 = q_n$. Let $\mc{E}$ be the splitting field, over $\mathbb{Q}$, of the polynomials $\{x^3 + a_nx + a_n : n \in \varnothing'\}$. Let $\omega_n$ be a primitive element for the splitting field of $x^3 + a_nx + a_n$, so that $\mc{E} = \mathbb{Q}(\omega_n : n \in \varnothing')$. Each of these polynomials has discriminant $D_n = -4a_n^3 - 27 a_n^2 = -a_n^2 q_n < 0$, and hence $\mathbb{Q}(\omega_n)$ has Galois group $S_3$. We claim that the Galois group of $\mc{E}$ is $\prod_{n \in \omega} S_3$. It suffices to show that given $m$ and $n_1,\ldots,n_\ell$ all distinct that $\mathbb{Q}(\omega_m)$ and $\mathbb{Q}(\omega_{n_1},\ldots,\omega_{n_\ell})$ are disjoint. Suppose not; then there is a non-trivial subfield $\mc{K}$ of $\mathbb{Q}(\omega_m)$ which is contained in $\mathbb{Q}(\omega_{n_1},\ldots,\omega_{n_\ell})$. We may assume that $\mc{K} = \mathbb{Q}(\sqrt{D_m}) = \mathbb{Q}(\sqrt{-q_m})$. Then $\sqrt{D_m} \in \mathbb{Q}(\sqrt{D_{n_1}},\ldots,\sqrt{D_{n_\ell}})$, a contradiction since $q_m,q_{n_1},\ldots,q_{n_\ell}$ are distinct primes. $\mc{E}$ does not have a splitting algorithm, but $\prod_{n \in \omega} S_3$ does not have a forking lattice of subgroups.

Now let $\mc{F} = \mathbb{Q}(\sqrt{-q_n} : n \in \varnothing')$. $\mc{F}$ does not have a splitting algorithm. By Theorem \ref{thm:CEAP}, $\mc{F}$ does not have the computable extension of isomorphisms property, and hence by Lemma \ref{lem:subfield}, $\mc{E}$ does not have the computable extension of automorphisms property.
\end{exam}

We do not know of any examples in which one cannot use either a direct application of Theorem \ref{thm:CEAP} or one of the methods in these two examples.

\section{Applications to Difference Closed Fields}
\label{Difference}

We will conclude this paper by applying our results to difference closed fields. The main idea will be to note that $(\mc{F},\sigma)$ embeds into a computable difference closed field if and only if there is an embedding $\imath$ of $\mc{F}$ into $\overline{\mc{F}}$ and an automorphism $\tau$ of $\overline{\mc{F}}$ such that $\tau$ $\imath$-extends $\sigma$. In the one direction, this will follow from an effective Henkin construction, while on the other hand it will follow from the fact that the algebraic closure of the prime field can be enumerated in any difference closed field.

\begin{thm}\label{thm:embedding-into-acfa}
Let $\mc{F}$ be a computable extension of $\mathbb{F}_p$, and $\sigma$ a computable automorphism of $\mc{F}$. Then the following are equivalent:
\begin{enumerate}
	\item $(\mc{F},\sigma)$ embeds computably into a computable difference closed field.
	\item There is a computable embedding $\imath \colon \mc{F} \to \overline{\mc{F}}$ of $\mc{F}$ into a computable presentation of its algebraic closure and a computable automorphism $\tau$ of $\overline{\mc{F}}$ which $\imath$-extends $\sigma$.
\end{enumerate}
\end{thm}
\begin{proof}
We begin by proving  (1)$\Rightarrow$(2). Suppose that there is a computable difference closed field $(\mc{K},\rho)$ into which $(\mc{F},\sigma)$ embeds. We can enumerate in $\mc{K}$ the algebraic closure $\overline{\mc{F}}$ of $\mc{F}$ (which is also the algebraic closure of the prime field) and the restriction $\tau$ of $\rho$ to $\overline{\mc{F}}$ (recall that every computable presentation of the algebraic closure of $\mc{F}$ is computable isomorphic to every other computable presentation). Then, since $(\mc{F},\sigma)$ embeds into $(\mc{K},\rho)$ and is algebraic over $\mathbb{F}_p$, its image is in $(\overline{\mc{F}},\tau)$. Then $\tau$ is an extension of $\sigma$ to $\overline{\mc{F}}$ via this embedding.

We now prove (1)$\Rightarrow$(2). The completions of $ACFA$ are given by the possible actions of the automorphism $\sigma$ on the algebraic closure of the prime field $\bar{\mathbb{F}}_p$ (see \cite[(1.4) of][]{ChatzidakisHrushovski99}). Let $\imath$ be a computable embedding of $\mc{F}$ into $\overline{\mc{F}}$ and $\tau$ an $\imath$-extension of $\sigma$ to $\overline{\mc{F}}$. Let $\mathcal{L}_{\overline{\mc{F}}}$ be the language of difference fields together with names for the constants of $\overline{\mc{F}}$. Let $T$ be the consistent theory axiomatized by $ACFA$ together with the existential diagram of $(\overline{\mc{F}},\sigma)$. Then $T$ contains a completion of $ACFA$, and since every formula is equivalent to an existential formula modulo $ACFA$, $T$ is complete. Moreover, $T$ is recursively axiomatizable and hence computable. So $T$ has a decidable model $(\mc{K},\rho)$. Using the embedding $\imath \colon \mc{F} \to \overline{\mc{F}}$, we get an embedding of the difference field $(\mc{F},\sigma)$ into $(\mc{K},\rho)$.
\end{proof}

We can use this, together with the examples from the previous section, to see that Rabin's Theorem on the existence of computable algebraic closures (and its analogue in differentially closed fields due to Harrington \cite{Harrington74}) does not hold in the context of difference closed fields:

\begin{corollary}
There exist computable difference fields which cannot be effectively embedded into any computable difference closed field.  Moreover, there is a counterexample in every characteristic.
\end{corollary}

\begin{proof} In characteristic zero, apply the previous corollary to the field from Proposition \ref{cor:forking-zero}, and in characteristic $p > 0$, by Corollary \ref{cor:forking-p}, we can use any normal extension of $\mathbb{F}_p$ with no splitting algorithm. \end{proof}

\begin{corollary} The analogue of Rabin's Theorem holds for difference fields with underlying field $\mc{F}$ if and only if $\mc{F}$ has the computable extension of automorphisms property. 
\end{corollary}


A set is \textit{low} if its Turing jump is as low as possible, i.e., Turing equivalent to $\varnothing'$. We note that every computable difference field does embed into a \textit{low} difference closed field:

\begin{fact}[essentially Friedman, Simpson, Smith \cite{FriedmanSimpsonSmith83}]
Every computable difference field embeds (by a map of low degree) into a low difference closed field.
\end{fact}
\begin{proof}
Let $(\mc{F},\sigma)$ be a computable difference field. Let $\imath \colon \mc{F} \to \overline{\mc{F}}$ be a computable embedding of $\mc{F}$ into its algebraic closure. Then there is a low automorphism $\tau$ of $\overline{\mc{F}}$ extending $\sigma$ (see \cite{FriedmanSimpsonSmith83}). The theory $ACFA$ together with the action of $\tau$ on $\overline{\mc{F}}$ is a complete low theory, and an effective Henkin construction produces a low model as in Theorem \ref{thm:embedding-into-acfa}.
\end{proof}

In Theorem \ref{thm:CEAP}, we showed that for a field whose Galois group has the non-covering property, having a splitting algorithm is equivalent to the computable extendability of automorphisms property. We do not know in general whether these are equivalent. We leave open:
\begin{ques}
For a normal extension $\mc{F}$ of $\mathbb{Q}$, is the computable extendability of automorphisms property equivalent to having a splitting algorithm?
\end{ques}


\bibliography{References}
\bibliographystyle{alpha}

\end{document}